\documentclass{amsart}

\usepackage{amsmath, latexsym,  amssymb, amsthm, amscd, comment, enumerate}

\usepackage[all]{xy}
\usepackage[british]{babel}
\usepackage{url}

\renewcommand{\sharp}{\#}

\renewcommand{\epsilon}{\varepsilon}

\DeclareMathOperator{\codim}{codim}
\DeclareMathOperator{\cod}{codim}
\DeclareMathOperator{\Gal}{Gal}

\newcommand{\torsione}{\zeta}

\newcommand{\sotto}{B}

\newcommand{\Vpi}{{V_\pi}}

\newcommand{\cVv}{(h(V)+\deg V)}
\newcommand{\cV}{(h(V)+\deg V)[k_\mathrm{tor}(V):k_\mathrm{tor}]}

\newcommand{\ccinque}{c_1}
\newcommand{\csei}{c_2}

\newcommand{\elle}{{\mathcal{L}}}

\newcommand{\qe}{\mathbb{Q}}

\newtheorem{thm}{Theorem}[section]

\newtheorem{con}[thm]{Conjecture}
\newtheorem{propo}[thm]{Proposition}
\newtheorem{lem}[thm]{Lemma}
\newtheorem{est}[thm]{Estimate}
\newtheorem{cor}[thm]{Corollary}
\newtheorem{D}[thm]{Definition}

\newtheorem{remark}[thm]{Remark}

\title[]{On the Torsion Anomalous Conjecture in CM abelian varieties}
\author{S. Checcoli}
\author{E. Viada}
\begin{document}

%\keywords{Subvarieties of products of elliptic curves, Finiteness of torsion anomalous intersections, Diophantine approximation}
\begin{abstract}
The Torsion Anomalous Conjecture (TAC) states that a subvariety $V$ of an abelian variety $A$ has only finitely many maximal torsion anomalous subvarieties. In this work we prove, with an effective method, some cases of the TAC when the ambient variety $A$ has CM, generalising our previous results in products of CM elliptic curves. When $V$ is a curve, we give new results and we deduce some implications on the effective Mordell-Lang Conjecture. 

\end{abstract}
\maketitle
\section{Introduction}
Let $A$ be an abelian variety embedded in the projective space and let $V$ be a proper subvariety of $A$. Assume that both $A$ and $V$ are defined over the algebraic numbers. 
\begin{D}  
The variety $V$  is a \emph{translate},  respectively  a \emph{torsion variety},  if it is a finite  union of translates of  proper algebraic subgroups by  points, respectively  by torsion points.

 $V$  is \emph{transverse}, respectively \emph{weak-transverse}, in $A$ if $V$ {is irreducible} and $V$ is not contained in any translate, respectively in any torsion subvariety of $A$.
\end{D}

It is a classical problem in diophantine geometry to investigate the relationship between the above geometrical definitions and the arithmetical properties of the variety $V$. In this direction, there are several celebrated theorems such as the so-called Manin-Mumford, Mordell-Lang and Bogomolov Conjectures.
\smallskip

Recently E. Bombieri, D. Masser and U. Zannier in \cite{BMZ1} suggested a new approach to this kind of investigations, introducing in particular the  notion of torsion anomalous intersections. 
\begin{D}
An irreducible subvariety $Y$ of $V$ is \emph{$V$-torsion anomalous} if 
\begin{itemize}
\item[-] $Y$ is an irreducible component of $V\cap (B+\zeta)$, with $B+\zeta$ an irreducible  torsion variety;
\item[-]  the dimension of $Y$ is larger than expected, i.e.  $${\mathrm {codim}\,} Y < {\mathrm {codim}\,} V +  {\mathrm {codim}\,}  B.$$
\end{itemize}

The variety $B+\zeta$ is \emph{minimal} for $Y$ if, in addition, it has minimal dimension.
The \emph{relative codimension} of $Y$ is the codimension of $Y$ in such a  minimal $B+\zeta$.

We say that $Y$ is {\emph{maximal} if it} is not contained in any $V$-torsion anomalous variety of strictly larger dimension.
\end{D}

In \cite{BMZ1}, the authors formulate several conjectures. We state here one natural variant.
\begin{con}[TAC, Torsion Anomalous Conjecture] For any algebraic subvariety  $V$  of a (semi-)abelian variety, there are only finitely many maximal $V$-torsion anomalous varieties.\end{con}
The TAC is well known to have several important consequences. It implies, for instance, the Manin-Mumford and the Mordell-Lang Conjectures; it is also related to model theory by work of B. Zilber and to algebraic dynamics by recent work of J. H. Silverman and P. Morton.  In addition, R. Pink generalised it  to mixed Shimura varieties.

Only few cases of the TAC are known: Viada \cite{ioant} proved it for curves in a product of elliptic curves; Maurin \cite{Maurin} for curves in a torus; Bombieri, Masser and Zannier \cite{BMZ1} for varieties of codimension  $2$ in a torus.  Habegger \cite{hab} gave related results, under some stronger assumptions on $V$.

In \cite{TAI}, we prove an effective TAC for maximal $V$-torsion anomalous varieties of relative codimension one in a product of CM elliptic curves.
Our bounds are explicit and uniform in their dependence on $V$.
As an immediate corollary, we prove the TAC for varieties of codimension 2, obtaining an elliptic analogue of the toric result in \cite{BMZ1}. 
In this work we generalise our results to CM abelian varieties. In \cite{TAI}, we also point out interesting relations between this kind of theorems and other relevant conjectures, such as the Zilber-Pink Conjecture and the above-mentioned ones. 
\smallskip

Let $A\subseteq \mathbb{P}^m$ be an abelian variety with CM defined over a number field $k$ and let $k_{\mathrm{tor}}$ be the field of definition of the torsion points of $A$.  Let $A$ be isogenous to a product of simple abelian varieties of dimension at most $g$. For a point $x\in A$, we denote by $\hat{h}(x)$ its canonical N\'eron-Tate height. For a subvariety $V\subseteq A$, we denote by $h(V)$ its normalised height and by $k_{\mathrm{tor}}(V)$ its field of definition  over $k_{\mathrm{tor}}$ (see Section \ref{Prelo}). By $\ll$ we denote an inequality up to a multiplicative constant depending  on $A$. Our main result is the following:
\begin{thm}\label{main}
 Let $V\subseteq A$ be a weak-transverse subvariety of codimension $>g$.  Then there are only finitely many maximal $V$-torsion anomalous  subvarieties $Y$ of relative codimension 1. 

{\bf Effective version: }More precisely, if $B+\zeta$ is minimal for $Y$, then for any positive real $\eta$  there exist constants depending only on $A$ and $\eta$ such that:
\begin{enumerate}
\item\label{non_tras} if $Y$  is not a translate then
\begin{align*}
 \deg B &\ll_\eta   (h(V)+ \deg V)^{\codim B+\eta},\\
h(Y)  &\ll_\eta   (h(V)+ \deg V)^{\codim B+\eta},\\
 \deg Y &\ll_\eta \deg V (h(V)+ \deg V)^{\codim B-1+\eta};
 \end{align*}
\item\label{punto}  if $Y$ is a point, then 
 \begin{align*}
 \deg B &\ll_\eta (\cV)^{\codim B+\eta} ,\\
\hat{h}(Y)  &\ll_\eta (h(V)+\deg V)^{\codim B+\eta}[k_{\mathrm{tor}}(V):k_{\mathrm{tor}}]^{\codim B -1+\eta},\\
[k_{\mathrm{tor}}(Y):k_{\mathrm{tor}}] &\ll_\eta \deg V[k_\mathrm{tor}(V):k_\mathrm{tor}]^{\codim B+\eta}{(h(V)+\deg V)}^{\codim B-1+\eta};
\end{align*}
\item\label{trass} if $Y$ is a translate of positive dimension, then
 \begin{align*}
  \deg B&\ll_\eta {(\cV)}^{\codim B+\eta},\\
 h(Y)&\ll_\eta {(h(V)+\deg V)}^{\codim B+\eta}{[k_\mathrm{tor}(V):k_\mathrm{tor}]}^{\codim B -1+\eta},\\
 \deg Y&\ll_\eta \deg V{(\cV)}^{\codim B-1+\eta}.
 \end{align*}
\end{enumerate}
 In addition,  the torsion points $\zeta$ belong to a  finite set of cardinality effectively bounded in terms of $\deg V$, $\deg B$ and constants depending only on $A$.
\end{thm}

{This theorem can be reformulated in the context of several other well known conjectures, as explained in the introduction of \cite{TAI}.}

The proof of Theorem \ref{main} is split in two parts, depending on whether $Y$ is a translate or not: {in Section \ref{seztanontraslati} we prove} part \eqref{non_tras} and in Section \ref{caso_tra} we prove parts \eqref{punto} and \eqref{trass}.

The main ingredients  {(see Section \ref{ingredienti})} needed for the proof of Theorem \ref{main} are: Zhang's inequality, the {Arithmetic} B\'ezout Theorem by P. Philippon, {our} sharp Bogomolov type bound proved in \cite{FuntorialeAV}, the {Relative} Lehmer by M. Carrizosa.
As usual, the CM hypothesis is due to the use of a Lehmer bound, known only for CM varieties. This result is only needed when $Y$ is a translate, while case \eqref{non_tras} of Theorem \ref{main} holds with the weaker assumption that $A$ has a positive density of ordinary primes, {as required} to apply a Bogomolov type bound (see \cite{galateau}, p.~779). 
{In particular,  our method could treat the case of general abelian varieties, if the Lehmer and Bogomolov type bounds were known in such generality.}

In Theorem \ref{curva}, proved in Section \ref{lacurva}, we expand our method in order to get some new effective results for curves in abelian varieties. This is particularly relevant, as bounds for the height in weak-transverse curves are  hard to obtain. For instance, such bounds allow us to  deduce some cases of the effective Mordell-Lang Conjecture, stated in Corollary \ref{MLR}.
The two classical  approaches to the effective Mordell-Lang Conjecture in abelian varieties are the Chabauty-Coleman and the Manin-Demjanenko methods. These methods require hypotheses that are similar to ours, but our result is of easier application {and more explicit}.
Finally in Section  \ref{dacurve}, we give some {generalisations} to varieties in abelian vareties.

In particular we prove the following results. We fix an isogeny of the CM abelian variety $A$ to the product 
  $\prod_{i=1}^\ell A_i^{e_i}$ of {non-isogenous} simple factors  $A_i$ of dimension $g_i$. Since isogenies preserve finiteness, without loss of generality, we identify $A$ with $\prod_{i=1}^\ell A_i^{e_i}$.
If $H\subseteq A$ is a subgroup, then $H=H_1\times \cdots \times H_\ell$, where $H_i\subseteq A_i^{e_i}$ is isogenous to $A_i^{f_i}$ for some $f_i\leq e_i$; therefore the matrix of the coefficients of the forms defining $H$ has a structure of a  block diagonal matrix, with entries in the endomorphism ring of the corresponding varieties.
{We can now state our effective result for weak-transverse curves, which is an example for the effective Zilber-Pink Conjecture.}
\begin{thm}
\label{curva}
 Let $C\subseteq A=\prod_{i=1}^\ell A_i^{e_i}$ be a weak-transverse curve. 
Then the set
$$\mathcal{S}(C)= C \cap \left(\bigcup_{H\in \mathcal{F}} H\right)$$ is a set of  bounded N\'eron-Tate height, where 
$\mathcal{F}$ is the family of all subgroups $H=\prod_{i=1}^\ell H_i\subseteq A$ such that \[\codim H_j> g_j \dim H_j\] for at least one index $j$ (here $\codim H_j$ is the codimension of $H_j$ in $A_j^{e_j}$).

More precisely, if $Y\in C\cap H$ then for any real $\eta>0$ there exists a constant, depending only on $A$  and $\eta$, such that
$$\hat{h}(Y)\ll_\eta (h(C)+\deg C)^{\frac{\codim H_j}{\codim H_j-g_j\dim H_j}+\eta}[k_{\mathrm{tor}}(C):k_{\mathrm{tor}}]^{\frac{g_j\dim H_j}{\codim H_j-g_j\dim H_j}+\eta}.$$
\end{thm}
To prove  Theorem \ref{curva} we first work in the projection  on the $j$-th factor of $A$, and then we lift the construction to the variety $A$.

\smallskip 

As an immediate consequence, we deduce the following corollary (proved in Section \ref{morlang}).

Let $\Gamma$ be a subgroup of $A=\prod_{i=1}^\ell A_i^{e_i}$. Assume that the group $\overline{\Gamma_i}<A_i$ generated by the {coordinates of the projections of $\Gamma$ on the factors $A_i^{e_i}$}  is an $\mathrm{End}(A_i)$ module of rank $t_i$. 
\begin{cor}
\label{MLR}
Let $A$ be a CM abelian variety and let $C$ be a weak-transverse curve in $A$. 
 Let $\Gamma$ be {a subgroup} as above, and suppose that
$t_j < \frac{e_j}{g_j +1}$
for some index $j$.
Then, for any positive $\eta$, there exists a  constant depending only on $A$ and $\eta$, such that  the set $C\cap \Gamma $ has  N\'eron-Tate height {bounded  as 
$$\hat{h}(C\cap \Gamma)  \ll_\eta (h(C)+\deg C)^{\frac{e_j-t_j}{e_j-(g_j+1)t_j}+\eta}[k_{\mathrm{tor}}(C):k_{\mathrm{tor}}]^{\frac{g_jt_j}{e_j-(g_j+1)t_j}+\eta}.$$}
\end{cor}
We remark that the corollary applies also to {some} $\Gamma$ of infinite rank, indeed we only assume that the rank on one projection is small (see  Remark \ref{rangoinfinito}).
%{In Remark \ref{rangoinfinito} we show how to use the Corollary to obtain bounds even for some special $\Gamma$ with infinite rank. }

\section{Preliminaries}\label{Prelo}

\subsection{Height and subgroups}\label{sgruppi}
We assume that all varieties are defined over the field of algebraic numbers.  

Let $A$ be an abelian variety {with CM}.  We fix, up to an isogeny, a decomposition of $A=\prod_{i=1}^\ell A_i^{e_i}$ in simple factors, of dimension $\dim A_i=g_i$. We consider an embedding $i_\elle$ of $A$ in $\mathbb{P}^m$ given by a symmetric  ample line bundle $\elle$ on $A$.  %defined over a number field. $\mathbb{K}$ of degree $\dk$. 
%We define $h(A)$ to be the Faltings height of $A$. 
Heights and degrees corresponding to $\elle$ are computed via $i_\elle$. More precisely, the degree of a subvariety of $A$ is the degree of its image under $i_\elle$; $\hat{h}=\hat{h}_\elle$  is the $\elle$-canonical N\'eron-Tate height of a point in $A$, and $h$ is the normalized height of a subvariety of $A$ as defined, for instance, in \cite{patriceI}. Notice that if $x\in A$ is a point, then $\hat{h}(x)=h(x)$.

By  Lemma 2.2 in \cite{Masserwustholz}, if $A$ is an abelian variety defined over a number field $k$, then every abelian subvariety of $A$ is defined over a finite extension of $k$ of degree bounded by $3^{16 (\dim A)^4}$; thus, without loss of generality, we assume that all abelian subvarieties of $A$ are defined over $k$.

Let $B+\zeta$ be an irreducible torsion variety of $A$.  Then $B=B_1\times \ldots \times B_l$ where $B_i\subseteq A_i^{e_i}$ is isogenous to $A_i^{f_i}$ for some integer $0\leq f_i\leq e_i$.
 
There exists a natural correspondence between abelian subvarieties $B$ of $A$, morphisms from $A$ to $\prod_{i=1}^\ell A_i^{e_i-f_i}$, and matrices made of $\ell$ blocks, where the $i$-th block is an $(e_i-f_i)\times e_i$-matrix with entries in the endomorphism ring of $A_i$. For details on such a correspondence see for instance \cite{FuntorialeAV}, Section 2.5. In short, the abelian subvariety $B$ defines the 
 projection morphism $\pi_B:A\rightarrow A/B$. The  successive minima of $B$  gives a matrix $\mathcal{H}_B$ of the above type. By multiplication on the left, the matrix $\mathcal{H}_B$  gives a morphism $\Phi_B$ from $A$ to $\prod_{i=1}^\ell A_i^{e_i-f_i}$, with $B$ the zero component of $\ker \Phi_B$. 
 
By Minkowski's theorem,  $\deg B$ is (up to constants depending only on $A$) the product of the squares of the norms of the
rows of $\mathcal{H}_B$. In addition, $B$ is the zero component of the   zero set of the forms $h_1,   \dots , h_r$ corresponding to the rows of $\mathcal{H}_B$. We order the $h_i$'s  by increasing {degrees}  $d_i$, {so that} $$d_1  \cdots d_r \ll \deg (B+\zeta)\ll d_1  \cdots d_r  .$$

We also recall that,  from \cite{Masserwustholz}, Lemma 1.3 and Lemma 1.4, if $B$ is an abelian  subvariety of $A$, and $B^\perp$  is its orthogonal complement, then $\deg B^\perp \ll \deg B$, and therefore $\sharp(B\cap B^\perp)\ll (\deg B)^2$.

\subsection{Torsion anomalous varieties}\label{pre}
 We recall some preliminary lemmas on torsion anomalous varieties used for our geometric constructions in the following sections. 
\begin{lem}
[Lemma 3.5 of \cite{TAI}]\label{wt}
Let $Y$ be a maximal $V$-torsion anomalous variety and  let $B+\zeta$  be minimal for $Y$. Then $Y$ is weak-transverse in $B+\zeta$ {(i.e. $Y$ is not contained in any  proper torsion subvariety of $B+\zeta$)}.\end{lem}

\begin{lem}[Lemma 3.6 of \cite{TAI}]\label{cruciale}
Let $Y$ be a maximal $V$-torsion anomalous variety, and let $B+\zeta$ be minimal for $Y$. Then $Y$ is a component of $V\cap (B'+\zeta)$ for every algebraic subgroup $B'\supseteq B$, with $\codim B'\ge \dim V-\dim Y$. 
\end{lem}

 The following lemma is due to Patrice Philippon \cite{Phil_preprint} and to certain properties of orthogonality in the Mordell-Weil groups studied by D. Bertrand \cite{BertOrto}.
 
We recall that the \emph{essential minimum} of a subvariety $X\subseteq A$ is defined as  \[\mu(X)=\sup\{\theta\in \mathbb{R}|\{ x \in X(\overline{\qe}) \mid \hat{h} (x)\le \theta\} \text{ is non dense in $X$}\}.\]

\begin{lem}[Philippon, Bertrand]\label{minimoess}
 Let $H+Y_0$ be a weak-transverse translate in  $A$, with  $Y_0$ a point in  the orthogonal complement $H^\perp$ of $H$. Then  $\mu(Y_0)= \mu(H+Y_0).$
 \end{lem}
 
We conclude with a remark on translations by torsion points. \begin{remark}\label{remzero}
Notice that, for any subvariety $X$ of $A$,  translations by a torsion point $\zeta$ leave invariant the degree, the field of definition over $k_{\rm{tor}}$ and the normalised height of $X$ (see also \cite{patriceI}, Proposition 9). In addition, if $Y\subseteq V\cap (B+\zeta)$ is $V$-torsion anomalous, then $Y-\zeta\subseteq (V-\zeta)\cap B$  is $(V-\zeta)$-torsion anomalous. Therefore, without loss of generality, we can work in $V$ or in $V-\zeta$, with the advantage, in this last case, that $B$ is an abelian subvariety.\end{remark}

\subsection{Main ingredients}\label{ingredienti}
We recall {here} the main ingredients used in the proof of Theorem \ref{main}. 
\subsubsection{The Zhang Estimate}\label{sezzhang}
The following theorem follows from the crucial result in Zhang's proof of the Bogomolov Conjecture (see %\cite{zhang} and 
\cite{ZhangEquidistribution}) and from the definition of normalized height.
\begin{thm}[Zhang]
Let $X\subseteq A$ be an irreducible subvariety.

%The obstruction index $\omega(V)$ is the minimal degree of a hypersurface containing $V$. 
%Define $\delta(V)$ as the minimal degree $\delta$ such that $V$ is, as a set, the intersection of hypersurfaces of degree $\leq\delta$. Then
%$
%\omega(V)\leq\delta(V)\leq\deg(V)\;.$
Then
\begin{equation*}%\label{zhang}
\mu(X) \le \frac{h(X)}{\deg X} \le (1+\dim X) \mu(X).
\end{equation*}
\end{thm}
\subsubsection{The Arithmetic B\'ezout Theorem}
The following version of the Arithmetic B\'ezout Theorem is due to Philippon \cite{patrice}.

\begin{thm}[Philippon]\label{bezout}
Let $X$ and $Y$ be irreducible subvarieties of the  projective space $\mathbb{P}^n$ defined over $\overline{\mathbb{Q}}$; {let} $Z_1, \dots , Z_g$ be the irreducible components of $X\cap Y$. Then 
$$ \sum_{i=1}^g h(Z_i)\le \deg X h(Y) +\deg Y h(X) +c(n) \deg X \deg Y,$$
where $c(n)$ is a constant depending only on $n$.
\end{thm}
%The constant $3 n^2$ is computed by Habegger in \cite{hab}, Theorem 3.

\subsubsection{An effective Bogomolov Estimate}
The following sharp Bogomolov bound is proved in  \cite{FuntorialeAV} and generalises a result of A. Galateau \cite{galateau}.
\begin{thm}[Checcoli, Veneziano, Viada] \label{due}
 Let $A$ be an abelian variety with a positive density of ordinary primes, and let $Y$ be an irreducible  subvariety of $A$  transverse in a translate $\sotto +p$. Then, for any $\eta>0$, there exists a positive constant $\ccinque$  depending on $A$ and $\eta$, such that
$$\mu (Y) \ge \ccinque  \frac{
(\deg \sotto)^{\frac{1}{\dim B-\dim Y}-\eta}
}{
(\deg Y)^{{\frac{1}{\dim B-\dim Y}}+ \eta}
}.$$
  \end{thm}

\subsubsection{A relative Lehmer Estimate}\label{sezcarri}
The following Lehmer bound is proved in \cite{carrizosaIMRN}.
\begin{thm}[Carrizosa]\label{carri2}
Let $A$ be an abelian variety with CM defined over a {number} field $k$, and let $k_\mathrm{tor}$ be the field of definition of all torsion points of $A$. Let $P$ be a point of infinite order in $A$, and let $B+\zeta$ be the torsion variety of minimal dimension containing $P$, with $B$ an abelian subvariety and $\zeta$ a torsion point. Then for every $\eta>0$ there exists a positive constant $\csei$ depending on $A$ and $\eta$, such that
\[
\hat{h}(P)\geq \csei  \frac{(\deg B)^{\frac{1}{\dim B}-\eta}}{[k_\mathrm{tor}(P):k_\mathrm{tor}]^{\frac{1}{\dim B}+ \eta}}.
\]
\end{thm}

\section{Finitely many maximal $V$-torsion anomalous varieties in $V\cap (B+\mathrm{Tor}_A)$}
Let $V$ be a weak-transverse variety in an abelian variety $A$. Let us fix an abelian subvariety $B$ of $A$. In  \cite{TAI}, Lemma 3.9. we proved that there are only finitely many $\zeta \in \mathrm{Tor}_{B^{\perp}} $  such that  $V \cap (B + \zeta)$ has a maximal $V$ -torsion anomalous component. In this section, Proposition \ref{sottogruppo},  we prove that the number of such $\zeta$ is effectively bounded in terms of $\deg V$, $\deg B$ and some constants depending on $A$. We thank the referee for pointing out the effectivity question and for his useful comments.

The proof of such an effective  result is based on an induction on the dimension of $V$,  on
 R\'emond's quantitative version of the Manin-Mumford Conjecture \cite{Rem00} and on the effective bound for the degree of the maximal translates in a variety implied, for instance,  by a result of Bombieri and Zannier \cite{BZ}.
We first recall these results and some other well known bounds. 

 Recall that  $A$ is abelian variety and  $\elle$ is a symmetric ample line bundle on $A$. We denote by $h_1(A)$ the projective height of the zero of $A$ in the embedding associated with $\elle^{\otimes 16}$ (as defined in \cite{DP2}, Notation 3.2) and by $d_A$  the degree of the field of definition of $A$. 
{If $G$ is an abelian subvariety of $A$ or a quotient of $A$, then $h_1(G)$ is bounded in terms of $h_1(A)$, $\deg A$, $\dim A$ and $\deg G$ (see \cite{DP2}, Proposition 3.9).}

Moreover, in several works, Masser and W\"ustholz, and then other authors, proved that 
for any abelian  subvariety $G$ of $A$, the degree of the field of definition of $G$ is at most $3^{16 ({\dim A})^4}d_A $ 
(see \cite{Masserwustholz} Lemma 2.2). 
Below, we  sum up these bounds.

\begin{est}\label{cG}  {If $G$ is an abelian subvariety of $A$ or a quotient of $A$, then:
\begin{itemize}
\item $d_G$ is bounded in terms of $d_A$ and $\dim A$;
\item $h_1(G)$ is bounded in terms of  $h_1(A)$, $\deg A$, $\dim A$ and $\deg G$.
\end{itemize}}
\end{est}

For simplicity, in what follows we shall denote by  $c(A)=c(\dim A, d_A, h_1(A),\deg A)$ any constant depending on $\dim A, d_A, h_1(A)$ and $\deg A$.              

We recall the  following  consequence of  R\'emond's result,  \cite{Rem00} Theorem 1.2. 
\begin{est}\label{lemrem}
The number of irreducible components of the closure of the torsion of a weak-transverse variety $V$ in an abelian variety $A$ is effectively bounded  as \begin{equation*}\label{b_rem}c(A)(\deg V)^{(\dim A)^{5(\dim V+1)^2}}.\end{equation*}  
\end{est}
Following the  work of R\'emond \cite{Rem00} Theorem 2.1 and the results in \cite{DP2} one sees that  if $G$ is an abelian subvariety of $A$ or a quotient of $A$, then {the corresponding constant $c(G)$ appearing in Estimate \ref{lemrem} is bounded  only in terms of $\dim A, d_A, h_1(A)$, $\deg A$ and $\deg G$. }

In our previous work  joint  with F. Veneziano \cite{TAI}, Lemma 7.4, we gave an explicit version of  a corollary of Lemma 2 in  \cite{BZ}. This is a  bound for the degree of the maximal translates contained in a variety, and so in particular for the degree of each component of the closure of the torsion.  More precisely:
\begin{est}\label{lemBZ}
If $V$ is weak-transverse in an abelian variety $A$, then the
maximal translates contained in $V$ have degree bounded by $c(A)(\deg V)^{2^{\dim V}}$.
 \end{est}
 Notice that, if $\zeta $ is a torsion point such that $V\cap (B+\zeta)$ has a $V$-torsion anomalous component, then all the points in $\zeta +(B\cap B^{\perp})$ share the same property. Indeed  $B+\zeta=B+\zeta+(B\cap B^{\perp})$. Clearly, we shall avoid such a redundancy and work  up to  points in $B\cap B^{\perp}$.   Nevertheless, $|B\cap B^{\perp}|\ll(\deg B)^2$.  This makes the following definition consistent.
 
 \begin{D}
Let $B$ be an abelian subvariety of an abelian variety $A$. Let $V$ be a  weak-transverse subvariety of $A$. We denote by $Z_{V,A}$ the set of torsion points $\zeta\in B^{\perp}/B\cap B^{\perp}$ such that  $V\cap (B+\zeta)$ has a maximal $V$-torsion anomalous component $Y_\zeta$. \end{D}
{We point out that  the set $Z_{V,A}$ also depends on  $B$. However, in our proof  $B$ is fixed, while $V$ and $A$ vary. To simplify the notation we only indicate the dependence on $V$ and $A$.}

In the following proposition we estimate the number of points in $Z_{V,A}$. The number of maximal $V$-torsion anomalous components in $V\cap (B+{\rm Tor}_A)$ is clearly estimated by $|Z_{V,A}|\deg V \deg B$.

\begin{propo}\label{sottogruppo}
 Let $\sotto$ be an abelian subvariety of an abelian variety $A$. Let $V$ be weak-transverse in $A$. Then the cardinality  of $Z_{V,A}$  is effectively bounded in terms of $\deg V, \deg B$ and constants  depending only on $\dim A, h_1(A), d_A$ and $\deg A$.
\end{propo}

\begin{proof}
Consider the  projection $$\pi_B:A\rightarrow A/B.$$

We recall that the degree of the image via $\pi_B$ of a variety $X\subseteq A$ and the degree of the preimage  via $\pi_B$ of a variety $X\subseteq A/B$ only depend on $\deg X$, $\deg B$ and $\deg A$. In particular, $\deg \pi_B(V)$ is bounded in terms of $\deg V$, $\deg B$, $\deg A$ and
 $\deg A/B$ is bounded in terms of $\deg B$ and $\deg A$.

The proof of our proposition is done by induction on the dimension of $V$.\\

The base of our induction is the case of a curve, i.e. $\dim V=1$. 
 Then $\pi_B(V)$ is a weak-transverse  curve in $A/ B$, because $V$ is weak-transverse in $A$. Moreover the points of $Z_{V,A}$ map to torsion points of $\pi_B(V)$.  The number of torsion points of $\pi_B(V)$ is estimated using  Estimate \ref{lemrem}. Their preimage, which contains  $Z_{V,A}$,  {has then cardinality} effectively bounded in terms of $\deg V$, $\deg B$ and $c(A)$.\\

 Suppose by  inductive hypothesis, that the proposition  holds for every variety $V$ with $\dim V<n$. We then show that it holds for  $V$   of dimension $n$. \\

To prove our result, we are going to partition $Z_{V,A}$ into a finite union of subsets $Z_{X}$ associated with irreducible subvarieties $X$ of $V$ of dimension $<n$. We then verify that such varieties $X$ satisfy the  assumption of the proposition; by the inductive hypothesis, we deduce that the cardinalities $|Z_{X}|$ are effectively bounded  in terms of $\deg V$, $\deg B$ and $c(A)$.

Denote by $f: V \to A/B$ the restriction of $\pi_B$ to $V$.

If $f$ is dominant, then the generic fiber $F_p=V\cap (B+\tilde{p})$   has dimension \begin{equation}\label{equidimf}\dim F_p=\dim V-\codim B\end{equation}
here $p$ belongs to an open subset of $A/B$  and $f(\tilde{p})=p$. The  dimensional equation (\ref{equidimf}) shows that the  generic fiber is not anomalous. Consider the subset $V_\pi$ of $A/B$ given by all points that do not have generic fiber. By the fiber dimension theorem (see for instance \cite{Sha}, Section 6.3, Theorem 7), this is a proper closed subset of $\pi_B(V)$ and its degree is effectively bounded in terms of $\deg V$, $\deg B$ and $c(A)$. Note that,  the image of $Z_{V,A}$ via $\pi_B$ is a subset of the torsion of $V_\pi$,  indeed the fiber of a point in $\pi_B(Z_{V,A})$ is torsion anomalous and therefore does not satisfy the equality (\ref{equidimf}). 

If $f$ is not dominant, then set $V_\pi=\pi_B(V)$. Clearly, $Z_{V,A}$ is a subset of the torsion of $V_\pi$.

Note that in both cases:
\begin{enumerate}[(a)]
\item\label{ia} $\deg \Vpi$ is bounded  in terms of $\deg V$, $\deg B$ and $c(A)$. \end{enumerate}

Let  $T_1,\ldots,T_r$ be the isolated components of the closure of the torsion of $\Vpi$ intersecting $\pi_B(Z_{V,A})$.
Clearly
\[Z_{V,A}=\bigcup_{i=1}^r (\pi_B^{-1}(T_i)\cap Z_{V,A})\]
and
\begin{equation*}\label{card}|Z_{V,A}|=\sum_{i=1}^r |(\pi_B^{-1}(T_i)\cap Z_{V,A})|.\end{equation*} 

From Estimate \ref{lemrem} and \eqref{ia}, the number $r$ is effectively bounded  in terms of $\deg V$, $\deg B$, $\deg A$ and $c(A)$. Thus we shall prove that, 
 for every $1\le i \le r$, the cardinality $|(\pi_B^{-1}(T_i)\cap Z_{V,A})|$  is  effectively bounded in terms of $\deg V, \deg B$ and $c(A)$.

Let $T$ be one of the above components.  Define $$W=\pi_B^{-1}(T)\cap V.$$

We have that:
\begin{enumerate}[(i)]

 \item\label{degW}  $\deg W$ is bounded  in terms of $\deg V$, $\deg B$ and $\deg A$. Indeed by B\'ezout's theorem, $\deg W\le \deg \pi_B^{-1}(T) \deg V$. By Estimate \ref{lemBZ}, $\deg T$ is bounded in terms of the degree and the dimension of $\Vpi$ and  thus, by (\ref{ia}), in terms of $\deg B$, $\deg V$  and $c(A)$.
 
 \item\label{dimm}  $\dim W<n$, because $V$ is weak-transverse in $A$ and so it is not contained in $\pi_B^{-1}(T)$.  
  
\item\label{iii} For $\zeta\in \pi_B^{-1}(T)\cap Z_{V,A}$, each maximal  $V$-torsion anomalous  component   $Y_\zeta$  of $V\cap (B+\zeta)$ is contained in $W$, indeed $\pi_B(Y_\zeta)=\pi_B(\zeta)\in T$. 
 \end{enumerate}
 By \eqref{iii}, the variety $W$ contains all the $Y_\zeta$ that we are counting, however $W$ is not necessarily irreducible. Therefore we cannot hope to use the inductive hypothesis on $W$ and we have to consider its irreducible components.

Let  $X_1,\ldots,X_s$ be the irreducible components of $W$.  For $\zeta\in \pi_B^{-1}(T)\cap Z_{V,A}$, we denote by $Y_\zeta$ any maximal  $V$-torsion anomalous component of $V\cap (B+\zeta)$. By \eqref{iii}, clearly each $Y_\zeta$  is contained in some $X_i$. We are going to count the number of  $Y_\zeta$ contained in each $X_i$. 

Denote by \[Z_{X_j}=\{\zeta\in \pi_B^{-1}(T)\cap Z_{V,A}\mid {X_j} \text{ contains  some } Y_\zeta\}/( B\cap B^{\perp}) .\]

Then
\[\pi_B^{-1}(T)\cap Z_{V,A}=\bigcup_{j=1}^s Z_{X_j}.\]

The number $s$  of irreducible components of $W$    is bounded by $\deg W$. Thus, by \eqref{degW},  $s$ is effectively bounded only in terms of $\deg V$, $\deg B$ and $c(A)$.

To conclude our proof  we are left to bound in an effective way the cardinality of each  $Z_{X}$, for $X$ running over all irreducible components of $W$. 

If $X$ does not contain any $Y_\zeta$, then  $ |Z_{X}|=0$.

If $X=Y_{\zeta_0}$ for some $\zeta_0\in Z_{V,A}$, then $ |Z_{X}|=1$.

 Suppose  that  $X$ strictly contains $Y_{\zeta_0}$ for some $\zeta_0\in Z_{V,A}$.  In this case we are going to show  that $|Z_X|\le |Z_{X-\zeta_0, \pi_B^{-1}(T)-\zeta_0}|$. Applying the  inductive hypothesis, we estimate $ |Z_{X-\zeta_0, \pi_B^{-1}(T)-\zeta_0}|$ in terms of $\deg V$, $\deg B$ and $c(A)$.

 We first verify that $X-\zeta_0$ in $\pi_B^{-1}(T)-\zeta_0$ satisfies the assumption of the inductive hypothesis, that is the assumption of the proposition with $\dim X<n$.
 Observe that we  need to translate by $\zeta_0$ in order to obtain ambient varieties which are abelian varieties.
 \begin{itemize}
\item  The variety $\pi_B^{-1}(T)-\zeta_0$ is an abelian variety  containing $B$.
Indeed $B+\zeta_0$ is a subvariety of $\pi_B^{-1}(T)$ and $\zeta_0\in \pi_B^{-1}(T)$.\\

\item The variety $X-\zeta_0$ is weak-transverse in $\pi_B^{-1}(T)-\zeta_0$. Equivalently, by Remark \ref{remzero}, we show that $X$ is weak-transverse in $\pi_B^{-1}(T)$. 
Since $Y_{\zeta_0}$ is a maximal $V$-torsion anomalous variety and $X$ strictly contains $Y_{\zeta_0}$, then $X$ cannot be $V$-torsion anomalous. Recall that $X$ is a component of $V\cap \pi_B^{-1}(T)$. Thus 
 \begin{equation}
  \label{formulacod}\dim \pi_B^{-1}(T)-\dim X= \dim A-\dim V .\end{equation}
     
 If $ X$ were not weak-transverse in $\pi_B^{-1}(T)$, then $X\subseteq B_1\cap V$ with $B_1\subsetneq \pi_B^{-1}(T)$ a torsion variety.  This contradicts relation (\ref{formulacod}).
 \item By \eqref{dimm},  $\dim X \le \dim W<n$.\\  
 
\end{itemize}
Thus, by inductive hypothesis, we get that: 
\[|Z_{X-\zeta_0, \pi_B^{-1}(T)-\zeta_0}| \text{ is effectively bounded in terms of } \deg X, \deg B  \text{ and }c( \pi_B^{-1}(T)).\]

We now show that by our construction $\deg X$ and $c( \pi_B^{-1}(T))$ only depend on $\deg V$, $\deg B$ and $A$.

 \begin{itemize}

 \item By \eqref{degW}, $\deg X \le \deg W$ is effectively bounded  in terms of $\deg V$, $\deg B$ and $c(A)$.

\item  By Estimate \ref{lemrem}, we know that  $\deg  T$ is effectively bounded  in terms of $\deg \Vpi$ and $\dim V$. Moreover,  by  \eqref{ia}, $\deg \Vpi$  is bounded in terms of $\deg V$, $\deg B$ and $c(A)$.  Finally,  Estimate \ref{cG} ensures that 
$d_{ \pi_B^{-1}(T)}$ and 
$h_1( \pi_B^{-1}(T))$ are effectively bounded in terms of $\deg V$, $\deg B$ and $c(A)$.  
 \end{itemize}
Therefore  
\[|Z_{X-\zeta_0, \pi_B^{-1}(T)-\zeta_0}| \text{ is effectively bounded in terms of } \deg V, \deg B  \text{ and }c( A).\]

We finally prove that: $$|Z_X|\le |Z_{X-\zeta_0, \pi_B^{-1}(T)-\zeta_0}|.$$
We shall show that for every maximal $V$-torsion anomalous variety $Y_\zeta\subseteq X$, the variety $Y_\zeta-\zeta_0$ is a maximal $(X-\zeta_0)$-torsion anomalous variety {in $\pi^{-1}(T)-\zeta_0$}.

Clearly  $Y_\zeta-\zeta_0\subseteq (X-\zeta_0)\cap (B+\zeta-\zeta_0)$. Since $Y_\zeta$ is $V$-torsion anomalous we have
$$\dim B-\dim Y_\zeta<\dim A- \dim V.$$
From this and  \eqref{formulacod}  we obtain 
$$\dim B-\dim Y_\zeta<\dim A-\dim V=\dim \pi_B^{-1}(T)-\dim X.$$
Thus $Y_\zeta-\zeta_0$ is a  $(X-\zeta_0)$-torsion anomalous variety.

  In addition  $Y_\zeta-\zeta_0$ is maximal: let $Y'\supset Y_\zeta-\zeta_0$ be a maximal   $(X-\zeta_0)$-torsion anomalous variety and let  $B'+\zeta'$ be minimal for $Y'$. From \eqref{formulacod}, we have $$\dim B'-\dim Y'<\dim \pi_B^{-1}(T)-\dim X=\dim A-\dim V.$$ Thus $Y'+\zeta_0\subseteq V\cap (B'+\zeta'+\zeta_0)$ is  $V$-torsion anomalous and containing $Y_{\zeta}$. The maximality of $Y_{\zeta}$ as $V$-torsion anomalous implies $Y'+\zeta_0= Y_\zeta$. \\

In conclusion, collecting all our bounds, we have proven that 
$|Z_{V,A}|$
is effectively bounded in terms of $\deg V$, $\deg B$ and $c(A)$.
 \end{proof}

\section{Non translate torsion anomalous varieties}\label{seztanontraslati}
\begin{proof}[{\bf Proof of Theorem \ref{main} part \eqref{non_tras}}]
Let $Y$ be a maximal  $V$-torsion anomalous  variety which is  not a translate, and so of positive dimension. Let $B+\zeta$ be minimal for $Y$.
% Then $Y$ is a component  of $V\cap (B+\zeta)$ with $B$ an abelian variety and $\zeta$ a torsion point. In addition $\codim Y<\codim %V+\codim B$.
{We use the Arithmetic B\'ezout Theorem and the Bogomolov bound to prove that $\deg B$ is bounded only in terms of  $V$ and $A$, then we} deduce the bounds  for ${h}(Y)$ and $\deg Y$.

 By Lemma \ref{wt} $Y$  is weak-transverse in $B+\zeta$, and by assumption $\dim B =\dim Y +1$; therefore $Y$ is transverse in $B+\zeta$. 
 Applying the Bogomolov estimate Theorem \ref{due} to $Y$ in $B+\zeta$ we get
\begin{equation}\label{sotto}\frac{(\deg B)^{1-\eta}}{(\deg Y)^{{1}+\eta}}\ll_\eta \mu(Y).\end{equation}

Let $h_1, \dots , h_r$ be the forms of increasing degrees $d_i$ such that   $B+\zeta$ is a component of their zero set. We have that $r\leq \cod B\leq rg$ {and}
\begin{equation}\label{gradi}
 d_1 \cdots d_r \ll \deg (B+\zeta)=\deg B  \ll  d_1 \cdots d_r.
\end{equation}

{Consider} the algebraic subgroup given by the first $h_1\cdots h_{r-1}$ forms,  {and let} $B'$ be {one of its irreducible components} containing $B+\zeta$. Then  by \eqref{gradi} we have $$\deg B'  \ll d_1 \cdots d_{r-1}\ll (\deg B)^\frac{r-1}{r}$$ 
 and $\codim B'\geq\cod B-g$.
 
Since $\codim V\geq g+1=g+\dim B-\dim Y$, this implies {that} $\codim B'\geq\dim V-\dim Y$, {and thus}, by Lemma \ref{cruciale},  $Y$ is a component of $V\cap B'$.

We apply the Arithmetic B\'ezout Theorem to $V\cap B'$ and recall that $h(B')=0$, because $B'$  is a torsion variety; {we get} 
\begin{equation}\label{sopra}h(Y) \ll (h(V)+ \deg V) \deg B'\ll \cVv (\deg B)^{\frac{r-1}{r}}.\end{equation} 
{Zhang's inequality, 
with \eqref{sotto} and \eqref{sopra}, gives}
$$\frac{{(\deg B)}^{1-\eta}}{(\deg Y)^{1+\eta}}\ll_\eta\mu(Y)\ll  (h(V)+ \deg V)\frac{ (\deg B)^\frac{r-1}{r}}{\deg Y}.$$

Recall that $Y$ is a component of $ V\cap (\sotto+\zeta)$. By B\'ezout's theorem $\deg Y \le \deg \sotto  \deg V$, thus 
\[{(\deg B)}^{1-\eta}\ll_\eta  (h(V)+\deg V)(\deg B)^\frac{r-1}{r}{(\deg \sotto \deg V)}^{\eta}\] and therefore \[{(\deg B)^{\frac{1}{r}-2\eta}}\ll_\eta  (h(V)+\deg V){ (\deg V)}^{\eta}.\]

For $\eta$ small enough  we get
 \begin{equation}
 \label{forte}
 \deg B \ll_\eta   (h(V)+ \deg V)^{r+\eta}(\deg V)^{\eta};
 \end{equation}
{this proves} that the degree of $B$ is bounded only in terms of $V$ and $A$. 
 Since the  abelian subvarieties of bounded degree are finitely many, applying Proposition \ref{sottogruppo}  we conclude that $\zeta$ belongs to a finite set of cardinality effectively bounded.
 
{The bound on the  height of $Y$ is now given} by  (\ref{sopra}) and (\ref{forte}) 
$$h(Y)\ll_\eta \cVv^{r +\eta}(\deg V)^{\eta}.$$

{Finally, the} bound on the degree is {obtained from \eqref{forte} and}  B\'ezout's theorem for the component $Y$ of $V\cap B'$:  \[\deg Y\ll_\eta \cVv^{r-1+\eta}(\deg V).\qedhere\]
 \end{proof}

\section{Torsion anomalous translates}\label{caso_tra}
\begin{comment}
\begin{thm}\label{tadimzero}
Let $A$ be an abelian variety with CM, and $g$ the maximal dimension of a simple abelian subvariety of $A$. Let $V\subseteq A$ be a weak-transverse subvariety of codimension $>g$. Then the maximal $V$-torsion anomalous varieties $Y$ of relative codimension 1 which are translates are finitely many. More precisely, let $k$ be a field of definition for $A$ and let $k_\mathrm{tor}$ be the field of definition of all torsion points of $A$.  Let also $Y$ be a maximal $V$-torsion anomalous variety and let $B+\zeta$ be minimal for $Y$. For any positive real $\eta$, there exist constants depending only on $A$ and $\eta$ such that:  
\begin{enumerate}
\item if $Y$ is a translate of positive dimension, then
 \begin{align*}
  \deg B&\ll_\eta {(\cV)}^{\codim B+\eta},\\
 h(Y)&\ll_\eta {(h(V)+\deg V)}^{\codim B+\eta}{[k_\mathrm{tor}(V):k_\mathrm{tor}]}^{\codim B -1+\eta},\\
 \deg Y&\ll_\eta (\deg V){(\cV)}^{\codim B-1+\eta};
 \end{align*}
 \item if $Y$ is a point, then 
 \begin{align*}
 \deg B &\ll_\eta (\cV)^{\codim B+\eta} ,\\
\hat{h}(Y_0)  &\ll_\eta (h(V)+\deg V)^{\codim B+\eta}[k_{\mathrm{tor}}(V):k_{\mathrm{tor}}]^{\codim B -1+\eta},\\
[k_{\mathrm{tor}}(Y_0):k_{\mathrm{tor}}] &\ll_\eta \deg V[k_\mathrm{tor}(V):k_\mathrm{tor}]^{\codim B+\eta}{(h(V)+\deg V)}^{\codim B-1+\eta}.
\end{align*}
\end{enumerate}
\end{thm}
\end{comment}
\begin{proof}[{\bf Proof of Theorem \ref{main}: parts \eqref{punto} and \eqref{trass}}]
Let $Y$ be a  maximal  $V$-torsion anomalous translate, with $\sotto+\zeta$ minimal for $Y$.

{We proceed to bound $\deg B$ and, in turn, the height and the degree of $Y$, using the Lehmer Estimate and the Arithmetic B\'ezout Theorem.}

The variety  $\sotto+\torsione$ is a component of the  torsion variety defined as the zero set of  forms $h_1,\dots,h_{r}$ of increasing degrees $d_i$, and $$d_1 \cdots d_{r}\ll\deg \sotto=\deg (\sotto+\torsione)\ll d_1 \cdots d_{r}.$$
We have that $r\leq \cod B\leq rg$.

Consider the torsion variety defined as the zero set of the first  $r-1$ forms $h_1, \dots, h_{r-1}$, and take  a connected component $B'$ containing $\sotto+\torsione$, {so that} 
$\deg B'\ll d_1 \cdots d_{r-1}\ll{(\deg B)}^{\frac{r-1}{r}}$
 and $\codim B'\geq\codim B -g.$

By Lemma \ref{cruciale},  Y  is a component of $V\cap B'$; indeed
\[\cod B'\geq \cod B-g=\dim A-g-\dim Y -1>\dim V-\dim Y-1.\]

\bigskip

The proof is now divided in two cases, depending on $\dim Y$. If $Y$ has dimension zero we use the Arithmetic B\'ezout Theorem and the Lehmer estimate; if $Y=H+Y_0$ is a translate of positive dimension, {we can reduce to the zero dimensional case using some properties of the essential minimum.}

{\bf Proof of part \eqref{punto}. }Consider first the case of a maximal torsion anomalous point $Y$.

All conjugates of $Y$ over $k_\mathrm{tor}(V)$ are {components of} $V\cap (B+\zeta)$, they all have the same normalised height and their number is at least  $$[k_\mathrm{tor}(V,Y):k_\mathrm{tor}(V)]\geq \frac{[k_\mathrm{tor}(Y):k_\mathrm{tor}]}{[k_\mathrm{tor}(V):k_\mathrm{tor}]}.$$
We then apply the Arithmetic B\'ezout Theorem in $V\cap B'$ obtaining
\begin{equation}\label{arbez61}
[k_\mathrm{tor}(Y):k_\mathrm{tor}]\hat{h}(Y)\ll {\cV}(\deg B)^{\frac{r-1}{r}}.
\end{equation}

 Applying Theorem \ref{carri2} to $Y$ in $B+\zeta$, we obtain that for every positive real $\eta$

\begin{equation}\label{car61}
\hat{h}(Y)\gg_\eta \frac{(\deg B)^{1-\eta}}{[k_{\mathrm{tor}}(Y):k_{\mathrm{tor}}]^{1+\eta}}.
\end{equation}

Combining \eqref{car61} and \eqref{arbez61}
we have
\begin{align*}
\frac{(\deg B)^{1-\eta}}{[k_{\mathrm{tor}}(Y):k_{\mathrm{tor}}]^{\eta}}&\ll_\eta[k_{\mathrm{tor}}(Y):k_{\mathrm{tor}}]\hat{h}(Y)\ll\\
&\ll {\cV}(\deg B)^{\frac{r-1}{r}}.
\end{align*}
For $\eta$ small enough we obtain
\begin{equation}\label{degB61}
\deg B\ll_\eta {(\cV)}^{r+\eta}[k_{\mathrm{tor}}(Y):k_{\mathrm{tor}}]^\eta.
\end{equation}
 Apply now B\'ezout's theorem to $V\cap B'$. All the conjugates of $Y$ over $k_\mathrm{tor}(V)$ are components of {this} intersection, so 
\begin{equation}\label{boundktor61}
\frac{[k_{\mathrm{tor}}(Y):k_{\mathrm{tor}}]}{[k_{\mathrm{tor}}(V):k_{\mathrm{tor}}]}\ll_\eta (\deg B)^{\frac{r-1}{r}} (\deg V).
\end{equation}
%\begin{equation}\label{boundktor61}
%\frac{[k_{\mathrm{tor}}(Y):k_{\mathrm{tor}}]}{[k_{\mathrm{tor}}(V):k_{\mathrm{tor}}]}\ll_\eta {(\cV)}^{r-1+\eta} (\deg V),
%\end{equation}
Substituting \eqref{degB61} into \eqref{boundktor61} we have the last bound of part (2) in the statement.

Finally, {we} apply  the Arithmetic B\'ezout Theorem to $V\cap B'$ to get
\begin{equation*}
\hat{h}(Y)\ll (h(V)+\deg V)(\deg B)^{\frac{r-1}{r}}\ll_\eta(h(V)+\deg V)^{r+\eta}[k_{\mathrm{tor}}(V):k_{\mathrm{tor}}]^{r-1+\eta}.
\end{equation*}

Having bounded $\deg B$, in view of Proposition \ref{sottogruppo} the points $\zeta$ belong to a finite set of cardinality effectively bounded.

\bigskip

{\bf Proof of part \eqref{trass}.} 
{Assume now that} $Y$ is a translate of positive dimension and write $Y=H+Y_0$, with $H$ an abelian variety and $Y_0$ a point in $H^\perp$.

{To bound $\deg B$ we can assume, without loss of generality, that $\zeta=0$ (see Remark \ref{remzero}).}

%$h(X+\zeta)=h(X), \deg (X+\zeta)=\deg X, [k_{tor}(X+\zeta):k_{tor}]= [k_{tor}(X):k_{tor}].$ 

 By Lemma \ref{minimoess},  
 \begin{equation}
 \label{A}\mu(Y_0)=\mu(H+Y_0).
 \end{equation} 
  
Since the intersection $V\cap B'$ is defined {over $k_\mathrm{tor}(V)$}, every conjugate of $H+Y_0$ {over $k_\mathrm{tor}(V)$} is a component of $V\cap B'$; as before, such components have the same normalised height and  their number is at least $$ \frac{[k_\mathrm{tor}(H+Y_0):k_\mathrm{tor}]}{[k_\mathrm{tor}(V):k_\mathrm{tor}]}.$$

  We apply the Arithmetic B\'ezout Theorem {in} $V\cap B'$ and we obtain 
 \begin{equation}\label{arbez742}
 {h}(H+Y_0)\frac{[k_\mathrm{tor}(H+Y_0):k_\mathrm{tor}]}{[k_\mathrm{tor}(V):k_\mathrm{tor}]}\ll (h(V)+\deg V) {(\deg B)}^{\frac{r-1}{r}}.
 \end{equation}
 By the Zhang's inequality, \eqref{A} and \eqref{arbez742}, we deduce 
 \begin{equation}\label{n7t}
 \mu(Y_0)\ll \frac{\cV {(\deg B)}^{\frac{r-1}{r}}}{[k_\mathrm{tor}(H+Y_0):k_\mathrm{tor}]\deg H}.
 \end{equation}
 
 %Note that Since $H+Y_0$ has relative codimension one and $Y_0\subseteq H^\perp_0\cap (B+\zeta)$ then $H'=H^\perp_0\cap (B+\zeta)$ has dimension one. 

%The relative codimension of $H+Y_0$ in $B+\zeta$ is one, therefore the  dimension of $H_0$ is one; moreover  $ H^\perp\cap (B+\zeta)$ has dimension one and it contains $Y_0$.  Consider the irreducible component $ H^\perp_0$ of the intersection containing $Y_0$: since $Y_0$ is not torsion, then $H^\perp_0$ has dimension one. So $H_0+\zeta= H^\perp_0$, because both varieties are irreducible, contain $Y_0$ and are one dimensional. 

%Notice {also} that $\zeta\in H_0^\perp \subseteq H^\perp$, and therefore $H_0\subseteq H^\perp$.

The lower bound for $\mu(Y_0)$ is derived as in  the case of dimension zero. 

Consider the smallest abelian subvariety $H_0$ of $B$ containing $Y_0$. Clearly $H_0$ is  the irreducible component of $ H^\perp\cap B$  containing $Y_0$. Indeed, they are both one-dimensional abelian varieties containing the point  $Y_0$ of infinite order.

By the definition of $H_0$, $B=H+H_0$,  and from \cite{Masserwustholz}, Lemma 1.2, we obtain
 \begin{equation}\label{t73gradiHH'}
 \sharp(H\cap H_0)\deg B  \leq \deg H \deg H_0.
 \end{equation}
 Moreover, from $H\cap H_0 \subseteq H\cap H^\perp$,  we get
 \begin{equation}\label{intH}
 \sharp(H\cap H_0)\leq \sharp(H\cap H^\perp)\ll (\deg H)^2.
 \end{equation}

 Applying Theorem \ref{carri2} to $Y_0$ in $H_0$ we get that, for every positive real $\eta$
 \begin{equation}\label{carrizosa74}
 \mu(Y_0)=\hat{h}(Y_0)\gg_\eta \frac{(\deg H_0)^{1-\eta}}{[k_\mathrm{tor}(Y_0):k_\mathrm{tor}]^{1+\eta}}.
 \end{equation}
 
 We {remark} that
 \begin{equation}\label{t73conj}
  [k_\mathrm{tor}(Y_0):k_\mathrm{tor}]\leq [k_\mathrm{tor}(H+Y_0):k_\mathrm{tor}]\cdot \sharp(H\cap H_0)
 \end{equation}
because if $\sigma\in\Gal(\overline{k_\mathrm{tor}}/k_\mathrm{tor}(H+Y_0))$ then $\sigma(Y_0)-Y_0\in H\cap H_0$, so $[k_\mathrm{tor}(Y_0): k_\mathrm{tor}(H+Y_0)]\leq  \sharp(H\cap H_0)$.

 Combining the upper bound and the lower bound for $\mu(Y_0)$ in \eqref{n7t} and \eqref{carrizosa74}, and using also \eqref{t73gradiHH'}, \eqref{intH} and \eqref{t73conj}, {for $\eta$ sufficiently small}, we have 
 \begin{equation}\label{bougradoB2} \deg B\ll_\eta {(\cV)}^{r+\eta}\end{equation}
where the dependence on $\deg H [k_\mathrm{tor}(H+Y_0):k_\mathrm{tor}]$ has been removed applying B\'ezout's theorem to the intersection $V\cap B'$.

 This also gives
\begin{equation*}
  \deg(H+Y_0) \ll_\eta (\deg V){(\cV)}^{r-1+\eta}.
 \end{equation*}
 
 Finally, from \eqref{arbez742}, \eqref{bougradoB2}  and the trivial bound $[k_\mathrm{tor}(H+Y_0):k_\mathrm{tor}]\geq 1$ we obtain
 \begin{equation*}\label{stimah74}
 h(H+Y_0)\ll_\eta {(h(V)+\deg V)}^{r+\eta}{[k_\mathrm{tor}(V):k_\mathrm{tor}]}^{r-1+\eta}.
 \end{equation*}
 
 Since we have bounded $\deg B$, we can conclude from Proposition \ref{sottogruppo} that the points $\zeta$ belong to a finite set of cardinality effectively bounded. 
\end{proof}

\section{The case of a curve and applications to the Effective Mordell-Lang Conjecture}\label{lacurva} 

{Recall that  $A=\prod_{i=1}^\ell A_i^{e_i}$, with $A_i$ non-isogenous simple CM factors of dimension $g_i$. To prove Theorem \ref{curva} we essentially follow the proof of Theorem \ref{main} part \eqref{punto}, working first in the projection {on} one factor, and then lifting the construction to the abelian variety $A$.}
\begin{proof}[Proof of Theorem \ref{curva}]
%We also use  \cite{ioannali} Proposition 4,  to bound  the degree of the field of definition  $Y_0$. 
 
%We notice that the set of all torsion points in $C$ is finite and has cardinality bounded by the effective Manin-Mumford Conjecture. From now on, we will be concerned with points in $\mathcal{S}(C)$ that are not torsion.

Clearly  all the points {in $ \mathcal{S}(C)$} are   $C$-torsion anomalous; {in addition, since $C$ is  a weak-transverse curve,} each torsion anomalous point is maximal. 

If $Y\in \mathcal{S}(C)$, then $Y\in C\cap H$, with $H=\prod_i H_i$ the  subgroup containing $Y$ which is minimal with respect to the inclusion.

Denote by $Y_i$ the projection of $Y$ on $H_i$ and by $C_i$ the projection of $C$ on $A_i^{e_i}$.
{Let $j$ be one of the index satisfying the hypothesis of the theorem}.
Assume {first} {that  $Y_j$} is a torsion point and define $H'=A_1^{e_1}\times \dotsb \times \{Y_j\}\times \dotsb \times A_\ell^{e_\ell}$. Clearly $\deg H'\ll 1$ and $h(H')=0$. Then, applying the Arithmetic B\'ezout Theorem to $Y$ in $C\cap H'$, we get
$
 \hat{h}(Y)\ll (h(C)+\deg C).
$

{Assume now}
that $Y_j$ is not a torsion point. Let $B_j+\zeta_j$ be a component of $H_j$ containing $Y_j$.
Clearly $\dim B_j=\dim H_j$ and $Y_j\in C_j\cap (B_j+\zeta_j)$ with $B_j+\zeta_j$ minimal for $Y_j$.
Furthermore, $Y_j$ is a component of $C_j\cap (B_j+\zeta_j)$ because $C_j$ is weak-transverse and, {by assumption,} $\codim H_j>g_j \dim H_j>0$.
{{This} ensures that the matrix {associated} to $B_j+\zeta_j$ has at least 2 {rows}, which is necessary to apply the method.

{We now sketch the proof which follows the proof of Theorem \ref{main} part \eqref{punto} and we give the relevant bounds.}}

%We first exclude the case $r=1$ and show that the case $N-r=1$ is covered by  Theorems \ref{tadimzero} and \ref{tadimzero2}. If $r=1$ then $\dim B<\codim B$ implies $\dim B=0$ and $N=\dim B+\codim B=1$ contradicting the weak-transversality of $C$.
% 

%
%Thus we can assume $r>N-r\geq 2$.
%Moreover $2r>N$, since by assumption $\codim B> \dim B$. 

The variety  $B_j+\zeta_j$ is a component of the zero set of  forms $h_1,\dots,h_{r}$ of increasing {degrees $d_j$ with} $$d_1 \cdots d_{r}\ll\deg B_j=\deg (B_j+\zeta_j)\ll d_1 \cdots d_{r},$$
{and we have} that $r={\codim B_j}/g_j=\codim H_j/g_j$.
%Notice that the case $N-b=1$ corresponds to $\dim B=1$ and $Y_0$ of relative codimension one, that can be treated with  Theorem \ref{main}.
%Moreover, $r+b>N$ because $N-b=\dim B<\frac{b}{g}\leq r$.

Consider the torsion variety defined as the zero set of $h_1$, and let $B_j'$ be  one of its connected components containing  $B_j+\zeta_j$; {then}
$
\deg B_j'\ll d_1\ll(\deg B_j)^{\frac{1}{r}}=(\deg B_j)^{\frac{g_j}{\codim B_j}}.$

From Theorem \ref{carri2} applied to $Y_j$ in $B_j+\zeta_j$, for every positive real $\eta$ we get
\begin{equation}\label{carrizosa81}
\hat{h}(Y_j)\gg_\eta \frac{(\deg B_j)^{\frac{1}{\dim B_j}-\eta}}{[k_{\mathrm{tor}}(Y_j):k_{\mathrm{tor}}]^{\frac{1}{\dim B_j}+\eta}}.
\end{equation}

Notice that all conjugates of $Y_j$ over $k_\mathrm{tor}(C_j)$ are components of $C_j\cap B_j'$ {and they all have the same height}.
Applying the Arithmetic B\'ezout Theorem to $C_j\cap B_j'$  and arguing as in the proof of Theorem \ref{main}, we have 
\begin{equation}\label{ArBez81}
\hat{h}(Y_j)\frac{[k_{\mathrm{tor}}(Y_j):k_{\mathrm{tor}}]}{[k_{\mathrm{tor}}(C_j):k_{\mathrm{tor}}]}\ll (h(C_j)+\deg C_j) (\deg B_j)^{\frac{g_j}{\codim B_j}}.
\end{equation}
 {Recall that $[k_{\mathrm{tor}}(Y_j):k_{\mathrm{tor}}]\geq 1$. }
From \eqref{carrizosa81} and \eqref{ArBez81} we get
$$
(\deg B_j)^{\frac{\codim B_j-g_j\dim B_j}{\codim B_j \dim B_j}-\eta}\ll_\eta [k_{\mathrm{tor}}(C_j):k_{\mathrm{tor}}] (h(C_j)+\deg C_j)   [k_{\mathrm{tor}}(Y_j):k_{\mathrm{tor}}]^{\frac{1}{\dim B_j}-1+\eta}.
$$
Since $\codim B_j-g_j\dim B_j=\codim H_j-g_j\dim H_j\geq 1$,  for $\eta$ {sufficiently} small {this yields}
\begin{equation}\label{degB81}
 \deg B_j\ll_\eta {([k_{\mathrm{tor}}(C_j):k_{\mathrm{tor}}] (h(C_j)+\deg C_j))}^{\frac{\codim H_j\dim H_j}{\codim H_j-g_j\dim H_j}+\eta},
\end{equation}
{where we use $ [k_{\mathrm{tor}}(Y_j):k_{\mathrm{tor}}]\geq 1$ if $\dim B_j>1$, and $[k_{\mathrm{tor}}(Y_j):k_{\mathrm{tor}}]\leq [k_{\mathrm{tor}}(C_j):k_{\mathrm{tor}}]\deg B \deg C_j$  if $\dim B_j=1$.}

We now lift the construction to $A$ as follows. Define $H'=A_1^{e_1}\times \dotsb \times B_j'\times\dotsb\times A_\ell^{e_\ell}.$
{Clearly $\deg H'\leq \deg A \deg B_j'$ and $Y$  is a component of} $C\cap H'$.
Applying the Arithmetic B\'ezout Theorem {to $C\cap H'$} and using \eqref{degB81}, we obtain 
\begin{multline}\label{sharper}
\hat{h}(Y)\ll (h(C)+\deg C)\deg H'\ll (h(C)+\deg C)(\deg B_j)^{\frac{g_j}{\codim H_j}} \\
\ll_\eta (h(C)+\deg C)^{\frac{\codim H_j}{\codim H_j-g_j\dim H_j}+\eta}[k_{\mathrm{tor}}(C):k_{\mathrm{tor}}]^{\frac{g_j\dim H_j}{\codim H_j-g_j\dim H_j}+\eta}.
\end{multline}
\end{proof}

\subsection{An application to the Effective Mordell-Lang Conjecture}\label{morlang}
%We now give an application of Theorem \ref{curva} to study some cases of the effective  Mordell-Lang conjecture.  
\begin{proof}[Proof of Corollary \ref{MLR}] 
%{Let $A$ be a CM abelian variety. We fix a representation of $A$ as a product of powers of simple {non-isogenous} abelian varieties %\[A=A_1^{e_1}\times \dotsb \times A_\ell^{e_\ell}\] with $d_i=\dim A_i$ and $d_1\leq \dotsb \leq d_\ell$ and we set $N=\dim A$.

%Let $\Gamma$ be a subgroup of $A$ and denote by $\Gamma_i$ {the projection of $\Gamma$ on the factor $A_i^{e_i}$ of $A$.}
%Suppose that the group $\overline{\Gamma}_i<A_i$
%=\langle \pi_B_1(\Gamma_i),\dotsc,\pi_{e_i}(\Gamma_i)\rangle<A_i$ 
%generated by the projections of $\Gamma_i$ on all coordinates of $A_i^{e_i}$ is an $\emor(A_i)$ module of rank $t_i$ for all $i$.
%}
%As usual we denote by $k$ a field of definition for $A$ and  by $k_{\mathrm{tor}}$ the field of definition of all torsion points of $A$. 
Let $x\in C\cap \Gamma$.
Let $j$ be an index such that $\frac{e_j}{g_j+1}>t_j$
and denote by $(x_1,\dots , x_{e_j})$ the projection of $x$ in $\Gamma_j$.

Let $\gamma_1, \dots , \gamma_{t_j}$ be  generators of the free part of $\overline{\Gamma}_j$.  Then there exist elements
 $0\neq a_k\in{\mathrm{End}}(A_j)$ for $k=1,\dotsc,e_j$, an $e_j\times t_j$ matrix $M_j$ with coefficients in  ${\mathrm{End}}(A_j)$ and a torsion point $\zeta\in A_j^{e_j}$ such that $$(a_1 x_1,\ldots,a_{e_j} x_{e_j})^{t}=M_j (\gamma_1,\ldots,\gamma_{t_j})^{t}+\zeta.$$  If the rank of $M_j$ is zero, then $(x_1,\dots , x_{e_j})$ is a torsion point and so it has height zero. 
 
If $M_j$ has positive rank $r_j$, we can choose $r_j$ equations of the system corresponding to $r_j$ linearly independent rows of $M_j$. We use these equations to write the $\gamma_k$ in terms of the $x_k$ and we substitute these expressions in the remaining equations. We obtain a system of maximal rank with $e_j-r_j\geq e_j-t_j$ linearly independent equations in the variables $x_1, \dots ,x_{e_j}$:
$$
 \begin{cases}
  m_{11}x_1+ \dots +m_{1e_j}x_{e_j}=\xi_1\\
  \vdots \\
  m_{e_j-r_j,1}x_1+ \dots +m_{e_j-r_j,e_j}x_{e_j}=\xi_{e_j-r_j}
\end{cases}$$
  where  $\xi_k\in A_j$ are  torsion points and $m_{k\ell}\in{\mathrm{End}}(A_j)$. These equations define a torsion variety $H_j\subseteq A_j^{e_j}$. {Since $(g_j+1)t_j<e_j$ we have}
  $\codim H_j>g_j \dim H_j. $

Then $x\in C\cap H$, where $H$ satisfies the hypothesis of Theorem \ref{curva}, which gives the bound for the height of $x$.
\end{proof}
\begin{comment}
\begin{remark}
Notice that, in order to get a better bound for the exponent $\delta$, instead of applying Theorem \ref{curva} as stated in the Introduction, we can apply the bound in equation \eqref{sharper}, with $r= \sum_{i=1}^\ell \max(e_i-t_i,0)$ as a bound for the number of equations, and $b=N-\sum_{i=1}^\ell d_i\min(e_i,t_i)$ as a bound for the codimension of $H$. Doing so allows to replace $\delta$ with the smaller constant $\delta'$:
\[\delta'=\frac{\sum_{i=1}^\ell d_i\min(e_i,t_i)}{\sum_{i=1}^\ell \max(e_i-t_i,0)-\sum_{i=1}^\ell d_i\min(e_i,t_i)}.\]
\end{remark}
\end{comment}
\begin{remark}\label{rangoinfinito}
Notice that it is possible to apply the corollary also to subgroups $\Gamma$ {whose  rank is bounded only on one projection.}

For example, let $E_1,E_2$ be two elliptic curves defined over $\qe$ and such that $E_1(\qe)$ is an abelian group of rank 1, and consider the product $A=E_1^4\times E_2$. 

%Let $\qe=k_0\subseteq k_1\subseteq k_2\subseteq \dotsb$ be a sequence of number fields such that $\bigcup_{j=0}^\infty k_j=\overline{\qe}$; %clearly the ranks of the groups $E_2(k_j)$ form an increasing sequence which tends to infinity.

{Let} $C$ be a weak-transverse curve in $A$. {Consider the subgroup} $\Gamma=E_1(\qe)^4\times E_2(\overline{\mathbb{Q}})$ {of} $A$.
%Applying the Corollary to these $C$ and $\Gamma_j$.
{Then} $\Gamma$ {is not of} finite rank, but, with the notation of the corollary, we have $g_1=1, e_1=4, t_1=1$ and
$ t_{1}<\frac{e_1}{g_1+1}=\frac{4}{2}.$

{The hypothesis of the corollary is therefore verified, and we have that
$$\hat{h}(C\cap \Gamma)  \ll_\eta (h(C)+\deg C)^{3/2+\eta}[k_\mathrm{tor}(C):k_\mathrm{tor}]^{1/2+\eta}.$$}
\end{remark}

\section{From curves to varieties}\label{dacurve}
We now adapt the proof strategy of Theorem \ref{main}  to obtain some new  results for varieties $V$ of dimension $>1$ embedded in a power $E^n$ of a CM elliptic curve.
%{The proof strategy of Theorem \ref{curva} can be generalised to varieties $V\subseteq A$ of dimension $>1$. For clarity, we give only the statement in a power $E^g$ of a CM elliptic curve; for general abelian varieties some technical conditions arise, which make a straightforward  generalisation  of little interest.
For subvarieties of general CM abelian varieties some technical condition arises. This makes a straightforward  generalisation  of our method of little interest.

For torsion anomalous varieties which are translates, the proof can be easily adapted, while for non translates a new argument is needed. {Indeed, in this last case,  the torsion anomalous variety is not transverse, but only weak-transverse in its minimal variety, a condition which is not sufficient to use the sharp Bogomolov bound.}

The torsion varieties contained in $V$ are already covered by the Manin-Mumford Conjecture, therefore we restrict ourselves to the $V$-torsion anomalous varieties which are {not torsion.}
\begin{thm}\label{t51}
Let $E$ be a CM elliptic curve defined over a number field $k$ and let $n>1$ be an integer.
Denote by $k_{\mathrm{tor}}$ a field of definition of all torsion points of $E$.

Let $V\subseteq E^n$ be a weak-transverse variety. Let $Y\subseteq V\cap B+\zeta$ be a maximal   $V$-torsion anomalous variety {which is not a torsion variety}, and let $B+\zeta$ be minimal for $Y$. 

Set $b=\dim B$, $v=\dim V$ and $y=\dim Y$ and assume that
$(n-b)>(v-y) (b-y).$

 Then for any $\eta>0$ there exist  constants depending only on $E^n$ and $\eta$ such that: 
\begin{enumerate}
\item\label{pu} if $Y$ is a point, then
\begin{align*}
\deg B &\ll_\eta   \left((h(V)+ \deg V)[k_{\mathrm{tor}}(V):k_{\mathrm{tor}}]\right)^{\frac{(n-b)b}{(n-b)-vb}+\eta},\\
\hat h(Y)  &\ll_\eta   (h(V)+ \deg V)^{\frac{n-b}{(n-b)-vb}+\eta} [k_{\mathrm{tor}}(V):k_{\mathrm{tor}}]^{\frac{vb}{(n-b)-vb}+\eta},\\
[k_{\mathrm{tor}}(Y):k_{\mathrm{tor}}] &\ll_\eta \deg V (h(V)+ \deg V)^{\frac{vb}{(n-b)-vb}+\eta} [k_\mathrm{tor}(V):k_\mathrm{tor}]^{\frac{n-b}{(n-b)-vb}+\eta};
\end{align*}
%\begin{itemize}
%\item[] $$ \deg B \ll_\eta   \left((h(V)+ \deg V)[k_{\mathrm{tor}}(V):k_{\mathrm{tor}}]\right)^{\frac{(g-b)b}{(g-b)-vb}+\eta}, $$
%\item[] $$\hat h(Y)  \ll_\eta   (h(V)+ \deg V)^{\frac{g-b}{(g-b)-vb}+\eta} [k_{\mathrm{tor}}(V):k_{\mathrm{tor}}]^{\frac{vb}{(g-b)-vb}+\eta}$$ 
%\item[]  $$ [k_{\mathrm{tor}}(Y):k_{\mathrm{tor}}] \ll_\eta \deg V (h(V)+ \deg V)^{\frac{vb}{(g-b)-vb}+\eta} [k_\mathrm{tor}(V):k_\mathrm{tor}]^{\frac{g-b}{(g-b)-vb}+\eta}.$$
%\smallskip
% \end{itemize}
\item\label{trapos} if $Y$ is a translate of positive dimension, then 
\begin{align*}
\deg B &\ll_\eta  (\cV)^{\frac{(n-b)(b-y)}{(n-b)-(v-y)(b-y)}+\eta},\\
h(Y)  &\ll_\eta {(h(V)+\deg V)}^{\frac{n-b}{(n-b)-(v-y)(b-y)}+\eta}{[k_\mathrm{tor}(V):k_\mathrm{tor}]}^{\frac{(v-y)(b-y)}{(n-b)-(v-y)(b-y)}+\eta},\\
\deg Y &\ll_\eta  \deg V(\cV)^{\frac{(v-y)(b-y)}{(n-b)-(v-y)}+\eta};
\end{align*}

%\begin{itemize}
%\item[] $$ \deg B\ll_\eta  (\cV)^{\frac{(g-b)(b-y)}{(g-b)-(v-y)(b-y)}+\eta}, $$
%\item[] $$h(Y)  \ll_\eta {(h(V)+\deg V)}^{\frac{g-b}{(g-b)-(v-y)(b-y)}+\eta}{[k_\mathrm{tor}(V):k_\mathrm{tor}]}^{\frac{(v-y)(b-y)}{(g-b)-(v-y)(b-y)}+\eta}, $$ 
%\item[]  $$ \deg Y \ll_\eta  (\deg V)^{\frac{(g-b)-(v-y)(b-y)}{(g-b)-(v-y)}+\eta} (\cV)^{\frac{(v-y)(b-y)}{(g-b)-(v-y)}+\eta}.$$
%\smallskip
% \end{itemize}
\item\label{ntra} if $Y$ is not a translate, then 
\begin{align*}
\deg B &\ll_\eta  \left((h(V)+\deg V) [k_\mathrm{tor}(V):k_\mathrm{tor}]\right)^{\frac{(b-y)(n-b)}{(n-b)-(v-y)(b-y)}+\eta},\\
h(Y)  &\ll_\eta (h(V)+\deg V)^{\frac{(n-b)}{(n-b)-(v-y)(b-y)}+\eta} [k_\mathrm{tor}(V):k_\mathrm{tor}]^{\frac{(v-y)(b-y)}{(n-b)-(v-y)(b-y)}+\eta},\\
 \deg Y &\ll_\eta  \deg V \left((h(V)+\deg V) [k_\mathrm{tor}(V):k_\mathrm{tor}]\right)^{\frac{(v-y)(b-y)}{(n-b)-(v-y)(b-y)}+\eta}.
\end{align*}
%\begin{itemize}
%\item[] $$ \deg B\ll_\eta  \left((h(V)+\deg V) [k_\mathrm{tor}(V):k_\mathrm{tor}]\right)^{\frac{(b-y)(g-b)}{(g-b)-(v-y)(b-y)}+\eta},$$
%\item[] $$h(Y)  \ll_\eta \left((h(V)+\deg V) [k_\mathrm{tor}(V):k_\mathrm{tor}]\right)^{\frac{(g-b)}{(g-b)-(v-y)(b-y)}+\eta},$$ 
%\item[]  $$ \deg Y \ll_\eta  (\deg V) \left((h(V)+\deg V) [k_\mathrm{tor}(V):k_\mathrm{tor}]\right)^{\frac{(v-y)(b-y)}{(g-b)-(v-y)(b-y)}+\eta}.$$
% \end{itemize}
%\smallskip
\end{enumerate}
 In addition  the torsion points $\zeta$ belong to a  finite set.
\end{thm}

\begin{proof}[{\bf Proof of Theorem \ref{t51} part \eqref{pu}}]

Let $Y$ be a  maximal  $V$-torsion anomalous point, with $\sotto+\zeta$ minimal for $Y$.

We proceed to bound $\deg B$ and, in turn, the height of $Y$ and the degree of its field of definition. {To this aim we use the Lehmer bound in Theorem \ref{carri2} and the Arithmetic B\'ezout Theorem.}

%The variety  $\sotto+\torsione$ is a component of the  torsion variety defined as the zero set of  forms $h_1,\dots,h_{r}$ of %increasing degrees $d_i$ and $$d_1 \cdots d_{r}\ll\deg \sotto=\deg (\sotto+\torsione)\ll d_1 \cdots d_{r},$$
%with $r= \cod B.$

Let $v=\dim V$ and $b=\dim B$.
%, consider the torsion variety defined as the zero set of the first  $v$ forms $h_1, \dots, h_{v}$, and take  a connected component %$B'$ containing $\sotto+\torsione$.  Then 
%Notice that $b<r$ because the point $Y_0$ is torsion anomalous.
%and $$\codim B'\geq\codim B -g.$$
By Lemma \ref{cruciale},  $Y$  is a component of $V\cap B'$ where $B'$ is, like in the proof of Theorem \ref{main}, the zero component of the  torsion variety defined by the first $v$ rows  $h_1,\dots, h_v$ of the matrix of  $B$. Then $\cod B'=v$ and 
$
\deg B'\ll {(\deg B)}^{\frac{v}{n-b}}.$
%All conjugates of $Y_0$ over $k_\mathrm{tor}(V)$ are in $V\cap (B+\zeta)$, so the number of components of $V\cap B'$ is at least %$$[k_\mathrm{tor}(V,Y_0):k_\mathrm{tor}(V)]\geq \frac{[k_\mathrm{tor}(Y_0):k_\mathrm{tor}]}{[k_\mathrm{tor}(V):k_\mathrm{tor}]}.%$$

We apply the Arithmetic B\'ezout Theorem to $V\cap B'$ {to obtain}
\begin{equation}\label{arbez61uno}
\hat{h}(Y)\ll \frac{\cV}{[k_\mathrm{tor}(Y):k_\mathrm{tor}]}(\deg B)^{\frac{v}{n-b}}.
\end{equation}
 Applying {the Lehmer estimate in} Theorem \ref{carri2} to $Y$ in $B+\zeta$, {instead, we have} that for every positive real $\eta$
\begin{equation}\label{car61uno}
\hat{h}(Y)\gg_\eta \frac{(\deg B)^{\frac{1}{b}-\eta}}{[k_{\mathrm{tor}}(Y):k_{\mathrm{tor}}]^{\frac{1}{b}+\eta}}.
\end{equation}
From \eqref{arbez61uno} and \eqref{car61uno}, and for 
 $\eta$ small enough, we  get
the bound for $\deg B$ if $b>1$. {If $b=1$ we use B\'ezout's theorem to bound the factor $[k_{\mathrm{tor}}(Y):k_{\mathrm{tor}}]^\eta$ as $((\deg B)(\deg V)[k_\mathrm{tor}(V):k_\mathrm{tor}])^\eta$.}

We then apply B\'ezout's theorem in $V\cap B'$ to bound $[k_{\mathrm{tor}}(Y):k_{\mathrm{tor}}]$ and 
 the Arithmetic B\'ezout Theorem {in $V\cap B'$} to prove the bound for $\hat{h}(Y)$. 
{Finally, from Proposition \ref{sottogruppo} it follows that} the points $\zeta$ belong to a finite set of cardinality effectively bounded.
\end{proof}

\begin{proof}[{\bf Proof of Theorem \ref{t51} part \eqref{trapos}}]
Let $Y=H+Y_0$ be a  maximal  $V$-torsion anomalous translate of positive dimension, with minimal $\sotto+\zeta$; {assume also that} $Y_0\in H^\perp$.

We use {the} Lehmer bound and the Arithmetic B\'ezout Theorem to bound $\deg B$ and, in turn, {the height and the degree of $H+Y_0$.}
In view of Remark \ref{remzero}, without loss of generality we can assume  that $\zeta=0$.

Let $b=\dim B, v=\dim V$ and $y=\dim Y=\dim H$. 
Clearly $v-y<n-b$ because $Y$ is torsion anomalous. 

{As before, by Lemma \ref{cruciale} we have that} Y  is a component of $V\cap B'$, where $B'$ is an irreducible torsion variety with $\codim B'=v-y$ and $
\deg B'\ll {(\deg B)}^{\frac{v-y}{n-b}}.
$

%By ,  
 %\begin{equation}
 %\label{Auno}\mu(Y_0)=\mu(H+Y_0).
 %\end{equation} 
% We are going to use the Arithmetic B\'ezout Theorem to find an upper bound for $\mu(H+Y_0)$ and Theorem \ref{carri2} to find a %lower bound.

%Since the intersection $V\cap B'$ is defined over $k_\mathrm{tor}$, every conjugate of $H+Y_0$ over $k_\mathrm{tor}$ is a %component of $V\cap B'$ and all such components have the same normalized height.  
%Moreover, all conjugates of $H+Y_0$ over $k_\mathrm{tor}(V)$ are in $V\cap (B+\zeta)$. 
Arguing as usual on the conjugates of $H+Y_0$ {over  $k_\mathrm{tor}(V)$, we see that the} intersection $V\cap B'$ has at least $ [k_\mathrm{tor}(H+Y_0):k_\mathrm{tor}]/[k_\mathrm{tor}(V):k_\mathrm{tor}]$ components.
 
  We apply the Arithmetic B\'ezout Theorem {to the intersection $V\cap B'$, obtaining}
 \begin{equation}\label{arbez7421}
 {h}(H+Y_0)\ll (h(V)+\deg V) {(\deg B)}^{\frac{v-y}{n-b}}\frac{[k_\mathrm{tor}(V):k_\mathrm{tor}]}{[k_\mathrm{tor}(H+Y_0):k_\mathrm{tor}]}.
 \end{equation}
 By Zhang's inequality, Lemma \ref{minimoess} and  \eqref{arbez7421}, we deduce
 \begin{equation}\label{n7t1}
 \mu(H+Y_0)=\mu(Y_0)\ll \frac{\cV {(\deg B)}^{\frac{v-y}{n-b}}}{[k_\mathrm{tor}(H+Y_0):k_\mathrm{tor}]\deg H}.
 \end{equation}

 For the lower bound for $\mu(Y_0)$, the proof follows the case of $\dim Y=0$. 
{Let $H_0=H^\perp\cap B$. 
 By minimality of $B$ we have that  $H_0$
is a torsion variety of {minimal} dimension containing $Y_0$, thus 
\[\dim H_0=\dim H^\perp+\dim B - n=(n-y)+b-n=b-y\] 
 %indeed $H+Y_0\subseteq B+\zeta$, and then $H\subseteq B+\zeta-Y_0=B$ because $Y_0\in B+\zeta$. 

%\[H_0+Y_0=(H^\perp\cap B) +Y_0=(H^\perp+Y_0)\cap(B+Y_0)=H^\perp\cap (B+\zeta)\]
}
% indeed if $H_0'$ were such a variety with strictly smaller dimension, then $H+Y_0$ would be an anomalous component of $V\cap %(H+H_0')$, and $\dim (H+H_0')<\dim (B+\zeta)$, against the minimality of $B+\zeta$.
%Note that Since $H+Y_0$ has relative codimension one and $Y_0\subseteq H^\perp_0\cap (B+\zeta)$ then $H_0=H^\perp_0\cap (B+\zeta)$ has dimension one. 
% Consider the smallest abelian subvariety $H_0$ of $B$ containing $Y_0-\zeta$, so $Y_0 $ is not contained in any torsion subvariety of $H_0+\zeta$. The relative codimension of $H+Y_0$ in $B+\zeta$ is one therefore the  dimension of $H_0$ is one.  Moreover  $ H^\perp\cap (B+\zeta)$ has dimension one and it contains $Y_0$.  Consider the irreducible component $ H^\perp_0$ of the intersection containing $Y_0$: since $Y_0$ is not torsion, then $H^\perp_0$ has dimension one. So $H_0+\zeta= H^\perp_0$, because both varieties are irreducible, contain $Y_0$ and are one dimensional. Notice in particular that $\zeta\in H_0^\perp \subseteq H^\perp$, and therefore $H_0\subseteq H^\perp$.

{As in Theorem \ref{main} part \eqref{trass}, one} can easily see that
 \begin{equation}\label{t73conj1}
  [k_\mathrm{tor}(Y_0):k_\mathrm{tor}]\leq [k_\mathrm{tor}(H+Y_0):k_\mathrm{tor}]\cdot \sharp(H\cap H_0).
 \end{equation}
By the definition of $H_0$, $B=H+H_0$  and from \cite{Masserwustholz}, Lemma 1.2, we get
 \begin{equation}\label{t73gradiHH'1}
 \sharp(H\cap H_0)\deg B  \leq \deg H \deg H_0.
 \end{equation}
 In addition $H\cap H_0\subseteq H\cap H^\perp$ thus
 \begin{equation}\label{degH2}
 \sharp(H\cap H_0)\ll (\deg H)^2.
 \end{equation}
% In fact if $\sigma\in\Gal(\overline{k_\mathrm{tor}}/k_\mathrm{tor}(H+Y_0))$ then $\sigma(Y_0)-Y_0\in H$. 
% Since $H_0+\zeta=H_0+Y_0=H_0$, and it is defined over $k_\mathrm{tor}$, we have that $\sigma(Y_0)-Y_0\in H_0$ as well. Hence the %number of conjugates of $Y_0$ over $k_\mathrm{tor}(H+Y_0)$ is at most $\sharp(H\cap H_0)$.
Applying Theorem \ref{carri2} to $Y_0$ in $H_0$ we get that, for every positive real $\eta$
 \begin{equation}\label{carrizosa741}	
 \mu(Y_0)=\hat{h}(Y_0)\gg_\eta \frac{(\deg H_0)^{\frac{1}{ b-y}-\eta}}{[k_\mathrm{tor}(Y_0):k_\mathrm{tor}]^{\frac{1}{b-y}+\eta}}.
 \end{equation}
 
 {Combining  \eqref{n7t1} and \eqref{carrizosa741} we have 
\begin{equation*}
\frac{(\deg H_0)^{\frac{1}{ b-y}-\eta}}{[k_\mathrm{tor}(Y_0):k_\mathrm{tor}]^{\frac{1}{b-y}+\eta}} \ll_\eta \frac{\cV {(\deg B)}^{\frac{v-y}{n-b}}}{[k_\mathrm{tor}(H+Y_0):k_\mathrm{tor}]\deg H}
 \end{equation*}
and hence, using  \eqref{t73conj1}, \eqref{t73gradiHH'1}, \eqref{degH2} as it was done in Theorem \ref{main} part \eqref{trass}, we get the bound for $\deg B$; notice that, if $b-y>1$, the argument is in fact simpler, as we don't need to deal with the  $[k_\mathrm{tor}(Y_0):k_\mathrm{tor}]^\eta$ term. 
 }

% \begin{equation}\label{t73combining}
 %\begin{split}
 %(\deg H_0)^{\frac{1}{\codim_B Y}-\eta}\ll_\eta {\cV} \frac{(\deg B)^{\frac{b}{r}}}{\deg H}\cdot\\ \cdot\frac{\sharp(H\cap H_0)^{\frac{1}%{\codim_B Y}+\eta}}{[k_\mathrm{tor}(H+Y_0):k_\mathrm{tor}]^{1-\frac{1}{\codim_B Y}-\eta}}.
 %\end{split}
% \end{equation}
 % Thus, possibly changing $\eta$, we have 
%  \begin{equation*}
 %\notag (\deg B)^{\frac{1}{\codim_B Y}-\eta}\ll_\eta  (\cV) \frac{(\deg B)^\frac{b}{r} \sharp(H\cap H_0)^{\eta}}{(\deg H)^{1-\frac{1}%{\codim_B Y}+\eta}}.
% \end{equation*}
%so that
% \begin{align}\label{bougradoB} \deg B&\ll_\eta  \frac{(\cV)^{\frac{r\codim_B Y}{r-b\codim_B Y}+\eta}}{(\deg H)^{\frac{r(\codim_B Y %-1)}{(r-b\codim_B Y)}-\eta}}\ll_\eta\\
%\label{bougradoB2} &\ll_\eta (\cV)^{\frac{r\codim_B Y}{r-b\codim_B Y}+\eta}.
 %\end{align}
 
Having obtained a bound for $\deg B$, the degree of $H+Y_0$ can be bounded applying B\'ezout's theorem to the intersection $V\cap B'$ and using
$\deg B'\ll\deg B^\frac{v-y}{n-b}.$
% Using \eqref{bougradoB} this becomes
 %\begin{equation}
 % (\deg H)^{1+\frac{b(\codim_B Y -1)}{(r-b\codim_B Y)}-\eta}\ll_\eta \deg V (\cV)^{\frac{b\codim_B Y}{r-b\codim_B Y}+\eta}
 %\end{equation}
 %\begin{equation}
 % (\deg H)\ll_\eta (\deg V)^{\frac{r-b\codim_B Y}{r-b}+\eta} (\cV)^{\frac{b\codim_B Y}{r-b}+\eta}
 %\end{equation}
The bound for $ h(H+Y_0)$, instead, is derived from \eqref{arbez7421} and the bound for $\deg B$.
% \begin{equation*}
% h(H+Y_0)\ll_\eta {(h(V)+\deg V)}^{\frac{r}{r-b\codim_B Y}+\eta}{[k_\mathrm{tor}(V):k_\mathrm{tor}]}^{\frac{b\codim_B Y}{r-%b\codim_B Y}+\eta}.
% \end{equation*}
Finally, from Proposition \ref{sottogruppo} we conclude that the points $\zeta$ belong to a finite set of cardinality effectively bounded. 
\end{proof}
\begin{proof}[{\bf Proof of Theorem \ref{t51} part \eqref{ntra}}] 
Assume that $Y\subseteq V\cap (B+\zeta)$ is not a translate.

 If $Y$ is transverse in $B+\zeta$ the proof of Theorem \ref{main} part \eqref{non_tras} easily adapts to this case as well, yielding the desired bounds; let us then assume that $Y$ is not transverse.  Without loss of generality, we can assume $\zeta=0$ (see Remark \ref{remzero}). Then $Y$ is transverse in a translate $H_1+Y_0\subsetneq B$, with $Y_0\in H_1^{\perp}$ and $H_1$ of minimal dimension. 

We define  $H_0=B\cap H_1^{\perp}$, so that  $B=H_1+H_0$ and 
\begin{equation}\label{degBsomma}\deg B=\deg (H_1+H_0)\leq \frac{\deg H_1\deg H_0}{\sharp(H_1\cap H_0)}.\end{equation} 

We set $y=\dim Y$, $v=\dim V$, $b=\dim B$, $h_1=\dim H_1$ and $h_0=\dim H_0=b-h_1$.

Writing $Y=Y_1+Y_0$ we have that $Y_1\subseteq H_1$ is transverse in $H_1$, because $Y$ is transverse in $H_1+Y_0$, and $Y_0\subseteq H_0$ is  transverse in $H_0$, because $B$ is minimal for $Y$.

By definition $Y_1\subseteq H_1$ and $Y_0\in H_1^{\perp}$. From Lemma \ref{minimoess} and  the  definition of essential minimum, we get \[\mu(Y)=\mu(Y_1)+\hat{h}(Y_0).\]

  As usual, the upper bound for $\mu(Y)$ is obtained using the Arithmetic B\'ezout Theorem  in $V\cap B'$ for some abelian variety $B'$, constructed deleting $v-y$ suitable rows from $B$. All conjugates of $Y$ are components of same height in $V\cap B'$. This gives 
  \begin{equation}\label{ancorabezu}
\mu(Y)\ll (h(V)+\deg V) (\deg B)^\frac{v-y}{n-b} \frac{[k_\mathrm{tor}(V):k_\mathrm{tor}]}{\deg Y[k_\mathrm{tor}(Y_1+Y_0):k_\mathrm{tor}]}.
\end{equation}

 Moreover 
\begin{equation}\label{confr} [k_\mathrm{tor}(Y_1+Y_0):k_\mathrm{tor}]\sharp(H_1\cap H_0)\geq [k_\mathrm{tor}(Y_0):k_\mathrm{tor}]\end{equation}  because for every $\sigma\in \mathrm{Gal}(\overline{k_\mathrm{tor}}/k_\mathrm{tor})$ which fixes $Y_1+Y_0$, the difference  $\sigma(Y_0)-Y_0$  lies in $H_1\cap H_0$.

To obtain a lower bound  for $\mu(Y)$ we either apply  the Bogomolov bound to $Y_1$ in $H_1$ or the Lehmer estimate to $Y_0$ in $H_0$. These give
\begin{equation}\label{bogoparte3casoa}
\frac{{(\deg H_1)}^{\frac{1}{h_1-y}-\eta}}{{(\deg Y)}^{\frac{1}{h_1-y}+\eta}}\ll_\eta \mu(Y_1)\leq \mu(Y)
\end{equation}
and 
\begin{equation}\label{BL2}\frac{(\deg H_0)^{\frac{1}{h_0}-\eta}}{[k_\mathrm{tor}(Y_0):k_\mathrm{tor}]^{\frac{1}{h_0}+\eta}}\ll_{\eta}\hat h(Y_0)\leq \mu(Y).\end{equation}

We now relate the left hand side to $\deg B$. Notice that one of the following relation holds: either 
\begin{equation}
\tag{i}(\deg B)^{\frac{h_1-y}{b-y}}< \deg H_1
\label{casoa}
\end{equation}
or  
\begin{equation}
\tag{ii} 
(\deg B)^\frac{h_0}{b-y}\leq \frac{\deg H_0}{\sharp(H_1\cap H_0)}.
\label{casob}
\end{equation}

 Indeed if \eqref{casoa} and \eqref{casob} were both false, then  \[\deg B=(\deg B)^{\frac{h_1-y}{b-y}+\frac{h_0}{b-y}}>\frac{\deg H_1\deg H_0}{\sharp(H_1\cap H_0)},\]
 which contraddicts  \eqref{degBsomma}.

Assume that (i) holds. Then (\ref{ancorabezu}), (\ref{bogoparte3casoa}) and (\ref{casoa}) and the fact that $n-b> (v-y)(b-y)$ give the bound 
\[{\deg B}\ll_\eta  \left((h(V)+\deg V) [k_\mathrm{tor}(V):k_\mathrm{tor}]\right)^{\frac{(b-y)(n-b)}{(n-b)-(v-y)(b-y)}+\eta},\]
where, if $h_1-y=1$, the factor $(\deg Y)^\eta$ has been removed applying  B\'ezout's theorem to $Y$ in $V\cap B$ and changing $\eta$ .

Assume that  \eqref{casob}  holds. Then (\ref{ancorabezu}), (\ref{BL2}) and (\ref{casob}), the fact that $n-b> (v-y)(b-y)$ and \eqref{confr} give the bound 
\[{\deg B}\ll_\eta  \left((h(V)+\deg V) [k_\mathrm{tor}(V):k_\mathrm{tor}]\right)^{\frac{(b-y)(n-b)}{(n-b)-(v-y)(b-y)}+\eta}\]
where, if $h_0=1$  the dependence on $[k_\mathrm{tor}(Y_0):k_\mathrm{tor}]$ can be removed using \eqref{confr}, bounding $[k_\mathrm{tor}(Y):k_\mathrm{tor}]$ as $[k_\mathrm{tor}(V):k_\mathrm{tor}]\deg V\deg B$ by the B\'ezout theorem applied to $Y$ in $V\cap B$  and observing that $\sharp(H_1\cap H_0)\leq \sharp(H_1\cap H_1^\perp)\ll (\deg H_1)^{2}\ll (\deg Y)^{2h_1}\le( \deg V \deg B)^{2h_1}$ because, since $Y_1$ is transverse in $H_1$, $H_1=Y_1+\cdots+Y_1$ ($h_1$ times), from which $\deg H_1\ll (\deg Y)^{h_1}$.

So we have bounded $\deg B$. 
We  obtain the bounds for $\deg Y$ and $h(Y)$ applying respectively the B\'ezout Theorem and the Arithmetic B\'ezout Theorem to the intersection $Y\subseteq V\cap B'$.
Finally,  Proposition \ref{sottogruppo} guarantees that the points $\zeta$ belong to a finite set of cardinality effectively bounded.

\end{proof}

\section*{Acknowledgments}
We  thank Francesco Veneziano for an accurate revision of an earlier version of this paper. We kindly thank the referee for his useful comments and corrections. Especially, we thank him  for pointing out the effectivity question of Proposition \ref{sottogruppo}. 
%{\bf AGGIUNGERE RINGRAZIAMENTI}

\def\cprime{$'$}
\providecommand{\bysame}{\leavevmode\hbox to3em{\hrulefill}\thinspace}
\providecommand{\MR}{\relax\ifhmode\unskip\space\fi MR }
\providecommand{\MRhref}[2]{%
  \href{http://www.ams.org/mathscinet-getitem?mr=#1}{#2}
}
\providecommand{\href}[2]{#2}

\end{document}